\documentclass[preprint,12pt]{elsarticle}

\usepackage{amssymb}
\usepackage{amsthm}

\journal{Journal of Algebra (Computational Alg. section)}

\usepackage{amsmath}
\usepackage{algorithm}
\usepackage{algorithmic}
\usepackage{multicol}
\usepackage{mathscinet}

\theoremstyle{plain}
\newtheorem{theo}{Theorem}[section]
\newtheorem{prop}[theo]{Proposition}
\newtheorem{lemma}[theo]{Lemma}
\newtheorem{cor}[theo]{Corollary}

\theoremstyle{definition}
\newtheorem{defi}[theo]{Definition}

\theoremstyle{remark}
\newtheorem{remark}[theo]{Remark}
\newtheorem{ex}[theo]{Example}

\DeclareMathOperator{\ann}{Ann}

\DeclareMathOperator{\lcm}{lcm}
\DeclareMathOperator{\ini}{in}
\DeclareMathOperator{\gr}{gr}
\DeclareMathOperator{\ord}{ord}
\def\K{\mathbb{K}}
\def\C{\mathbb{C}}
\def\R{\mathbb{R}}
\def\Z{\mathbb{Z}}
\def\N{\mathbb{N}}
\def\d{\partial}
\newcommand{\BS}{Bernstein-Sato }

\begin{document}

\begin{frontmatter}



\title{{\bf Algorithms for Checking Rational Roots of $b$-functions and their Applications}}


%
%

\author[viktor]{V.~Levandovskyy}
\address[viktor]{Lehrstuhl D f\"ur Mathematik, RWTH Aachen University, Templergraben 64, D-52062 Aachen, Germany}
\ead{viktor.levandovskyy@rwth-aachen.de}
\ead[url]{http://www.math.rwth-aachen.de/\~\,\! Viktor.Levandovskyy}

\author[jorge]{J.~Mart\'in-Morales\corref{cor1}\fnref{fn2}}
\cortext[cor1]{Corresponding author. Tel.: +34 976 76 13 24, Ext. 3275; Fax: +34 976 76 13 38.}
\fntext[fn2]{Partially supported by MTM2007-67908-C02-01, FQM-333 and ``E15 Grupo Consolidado Geometr\'ia (DGA)''}
\address[jorge]{Department of Mathematics-I.U.M.A., University of Zaragoza, C/Pedro Cerbuna, 12 - 50009, Zaragoza, Spain}
\ead{jorge@unizar.es}
\ead[url]{http://www.grupo.us.es/gmcedm}

\begin{abstract}
Bernstein-Sato polynomial of a hypersurface is an important object with numerous
applications. It is known, that it is complicated to obtain it computationally, as a number
of open questions and challenges indicate. In this paper we propose a family
of algorithms called \texttt{checkRoot} for optimized check of whether a given
rational number is a root of Bernstein-Sato polynomial and the computations of
its multiplicity.
This algorithms are used in the new approach 
to compute the whole global or local Bernstein-Sato polynomial
and $b$-function of a holonomic ideal with respect to weights. 
They are applied in numerous situations, where there is a possibility to compute an upper bound for the polynomial. Namely, it can be achieved by means of embedded resolution, for 
topologically equivalent singularities or using the formula of A'Campo and spectral numbers.
We also present approaches to the logarithmic comparison problem and 
the intersection homology D-module.
Several applications are presented as well as solutions to some challenges which were intractable with the classical methods.
One of the main applications consists of computing of a stratification of affine space
with the local $b$-function being constant on each stratum. Notably, the algorithm we propose 
does not employ primary decomposition. Also we apply our results for the computation of Bernstein-Sato polynomials for varieties.
The methods from this paper have been implemented in {\sc Singular:Plural} as libraries {\tt dmod.lib} and {\tt bfun.lib}. All the examples from the paper have been computed with this implementation.
\end{abstract}

\begin{keyword}
Singularities \sep $D$-modules \sep Bernstein-Sato polynomials \sep non-commutative Gr\"obner bases.

\MSC 16S32 \sep 68W30 \sep 32S40.
\end{keyword}

\end{frontmatter}




\section{Introduction}
\label{intro}

Through the article we assume $\K$ to be a field of characteristic 0.
By $R_n$ we denote the ring of polynomials $\K[x_1,\ldots,x_n]$ in $n$ variables over $\K$ and
by $D_n$ we denote the ring of $\K$-linear partial differential operators with coefficients in $R_n$, that is the
\textit{$n$-th Weyl algebra} \cite{MR01}. The ring $D_n$ is the associative $\K$-algebra
generated by the partial differential operators $\partial_i$ and the 
multiplication operators $x_i$ subject to relations
\[
\{ \partial_i x_j = x_j\partial_i + \delta_{i j}, x_j x_i = x_i x_j, \d_j \d_i = \d_i \d_j \mid 1\leq i,j \leq n \}. 
\]
That is, the only non-commuting pairs of variables are $(x_i,\d_i)$; they satisfy the
relation $\d_i x_i = x_i \d_i + 1$. We use the Lie bracket notation $[a,b]:=a b - b a$ for 
operators $a,b$, then e.g. the latter relation can be written as $[\d_i, x_i] = 1$.

Finally, we denote by $D_n[s]$ the ring of polynomials in one variable $s$ with coefficients in the
$n$-th Weyl algebra, i.e. $D_n[s] = D_n\otimes_\K \K[s]$.

Let us recall Bernstein's construction. Given a non-zero polynomial $f\in R_n$
in $n$ variables, we consider
$M = R_n[s,\frac{1}{f}]\cdot f^s$ which is by definition the free $R_n[s,\frac{1}{f}]$-module
of rank one generated by the formal symbol $f^s$. Then $M$
has a natural structure of left $D_n[s]$-module. Here the differential operators
act in a natural way,
\begin{equation}\label{action}
\partial_i(g(s,x)\cdot f^s) = \left(\frac{\partial g}{\partial x_i} + s g(s,x)
\frac{\partial f}{\partial x_i} \frac{1}{f} \right) \cdot f^s\ \in M
\end{equation}

\begin{theo}[Bernstein \cite{Bernstein72}]
For every polynomial $f\in R_n$ there exists a non-constant polynomial $b(s)\in \K[s]$
and a differential operator $P(s)\in D_n[s]$ such that
\begin{equation}\label{bernstein}
  P(s)f\cdot f^s = b(s)\cdot f^s\ \in \ R_n[s,\frac{1}{f}]\cdot f^s = M.
\end{equation}
\end{theo}

The monic polynomial $b(s)$ of minimal degree, satisfying (\ref{bernstein}) is called the \textbf{\BS polynomial} or the \textbf{global $b$-function}.



This paper is organized as follows. 
In Section~\ref{checkRootSection}, the \texttt{checkRoot} family of algorithms for checking rational roots of the global and local \BS polynomial is developed. We also show how to compute the $b$-function of a holonomic ideal with respect to a certain weight vector.
In Section~\ref{upperBound}, we show how to 
obtain an upper bound in various situations (by using an embedded resolution, for 
topologically equivalent singularities, by using A'Campo's formula and spectral numbers).
In particular, we demonstrate a complicated example of (non-isolated) quasi-ordinary singularity.

In Section~\ref{minRoot} we discuss the possibilities to obtain integral roots of the $b$-function
and apply it to the computation of the minimal integral root in the context of Intersection Homology $D$-module and Logarithmic Comparison Theorem. In Section~\ref{stratum}
we present a new method for computing the stratification of affine space, according to
local \BS polynomials. 

We want to stress, that \BS polynomials for most of the examples, presented in this paper, cannot be computed by direct methods with any computer algebra system including \textsc{Singular:Plural} \cite{Plural}. Indeed, these examples were known as open challenges in the community and here we present their solutions for the first time.

The timings for examples in this paper were performed on a PC with 4 Dual Core AMD Opteron 64 Processor 8220 (2800 MHz) (only one processor available during the computation) equipped with 32 GB RAM (at most 16 GB available during the computation) running openSUSE 11 Linux. 

\section{The \texttt{checkRoot} Family of Algorithms}
\label{checkRootSection}

For the sake of completeness, some of the ideas coming from \cite{LM08}, as well as some results and their proofs have been included here.

Several algorithms for computing the $b$-function associated with a polynomial are known, see e.g.~\cite{Oaku97a, Oaku97b, Oaku97c}, \cite{SST00}, \cite{BM02}, \cite{Noro02}, \cite{Schulze04}, \cite{LM08}. However, from the computational point of view it is very hard to obtain this polynomial in general. Despite significant recent progress, only restricted number of examples can be actually treated. In order to enhance the computation of the \BS polynomial via Gr\"obner bases,
we study the following computational problems.

\begin{enumerate}
\item Obtain an upper bound for $b_f(s)$, that is, find $B(s)\in \K[s]$ such that $b_f(s)$ divides $B(s)$.
$$
  B(s) = \prod_{i=1}^d (s-\alpha_i)^{m_i}
$$
\item Check whether $\alpha_i$ is a root of the $b$-function.
\item Compute the multiplicity of $\alpha_i$ as a root of $b_f(s)$.
\end{enumerate}

There exist some well-known methods to obtain an upper bound for the Bernstein-Sato polynomial of a hypersurface singularity once we know, for instance, an embedded resolution of such singularity \cite{Kashiwara76/77}. However, as far as we know, there is no algorithm for computing the $b$-function from this upper bound. In this section we present algorithms for checking whether a given rational number is a root of the $b$-function and for computing its multiplicity. As a first application, using this idea, we could obtain $b_f(s)$ for some interesting non-isolated singularities, see Example \ref{exCheckRoot} below.

From the definition of the $b$-function it is clear that
\begin{equation}\label{bernstein2}
  \langle b_f(s) \rangle = (\ann_{D_n[s]}(f^s) + \langle f \rangle)\cap \K[s].
\end{equation}
In fact, this is another way of defining the \BS polynomial. This equation was used to prove the main result of this section, namely Theorem \ref{mainresult}.

\begin{theo}\label{mainresult}
Let $R$ be a $\K$-algebra, whose center contains $\K[s]$. Let $q(s)\in \K[s]$ be a polynomial in one variable and $I$ a left ideal in $R$ satisfying $I\cap \K[s]\neq 0$. The following equalities hold:
\begin{enumerate}
\item $\big( I + R \langle q(s) \rangle \big)\cap \K[s] = I\cap \K[s] + \K[s] \langle q(s) \rangle$,
\item $\big( I : q(s)\big)\cap \K[s] = \big(I\cap \K[s]\big) : q(s)$,
\item $\big( I : q(s)^\infty\big)\cap \K[s] = \big(I\cap \K[s]\big) : q(s)^\infty$.
\end{enumerate}
In particular, using $I = \ann_{D_n[s]}(f^s) + \langle f \rangle \subseteq D_n[s]$ in the above equation (\ref{bernstein2}), we have
\begin{itemize}
\item $\big[\ann_{D_n[s]}(f^s)+ D_n[s] \langle f, q(s)\rangle \big] \cap \K[s]
  = \langle b_f(s),q(s)\rangle =  \big\langle\, \gcd(b_f(s),q(s))\, \big\rangle$,
\item $\big[(\ann_{D_n[s]}(f^s)+ D_n[s] \langle f \rangle) : q(s)\big] \cap \K[s]
  = \langle b_f(s)\rangle : q(s) =  \big\langle\, \frac{b_f(s)}{\gcd(b_f(s),q(s))}\,
 \big\rangle$,
\item $\big[ (\ann_{D_n[s]}(f^s)+ D_n[s] \langle f \rangle) : q(s)^\infty \big] \cap \K[s]
  = \langle b_f(s)\rangle : q(s)^\infty$.
\end{itemize}
\end{theo}

\begin{proof}
Let $b(s)\neq 0$ be a generator of $I\cap\K[s]$. At first, suppose that $h(s)\in (I+ R\langle q(s)\rangle)\cap \K[s]$. Then we have
\begin{equation}\label{h(s)}
  h(s) = P(s)+Q(s)q(s)
\end{equation}
where $P(s)\in I$ and $Q(s)\in R$. Let $d(s)$ be the greatest common divisor of $b(s)$ and $q(s)$. There exist $b_1(s)$ and $q_1(s)$ such that $d(s)b_1(s)=b(s)$ and $d(s)q_1(s)=q(s)$, and hence $b_1(s)q(s) = q_1(s)b(s)$. Since $s$ commutes with all elements in $R$,
multiplying the equation (\ref{h(s)}) by $b_1(s)$, one obtains
\begin{eqnarray*}
  b_1(s)h(s) &=&
  b_1(s)P(s) + Q(s)q_1(s)b(s) \ \in\ I
\end{eqnarray*}
Thus, $b_1(s)h(s)\in I\cap \K[s] = \langle b(s) \rangle$ and therefore $h(s)\in\langle b(s) \rangle : \langle b_1(s) \rangle = \langle d(s) \rangle = I\cap \K[s] + \langle q(s) \rangle$. The other inclusion follows obviously. The second and the third parts can be shown directly and the proof is complete.
\end{proof}

%

Note that the second (resp. third) part of the previous theorem can be used to heuristically find
an upper bound for $b_f(s)$ (resp. the roots of $b_f(s)$).
Since $q(s)$ is in the center of $D_n[s]$, the quotient and saturation ideals 
can be computed effectively via the kernel of a module homomorphism procedures, cf. \cite{LV3}. More classical but less effective approach is to use the extra commutative variable, say $T$, 
and
the formula
$$
  I:q(s)^\infty = D_n[s,T] \langle I, 1-T q(s) \rangle \cap D_n[s].
$$

Let us see an example to illustrate how useful could be Theorem \ref{mainresult}.

\begin{ex}
Let $f\in\C[x,y]$ be the polynomial $x(x^2+y^3)$. The annihilator of $f^s$ in $D[s]$ can be generated by the operators $P_1(s)= 3x y^2\partial_x-y^3\partial_y-3 x^2\partial_y$ and $P_2(s) = 3x\partial_x+2y\partial_y-9s$. Consider the univariate polynomial
$$
  q(s) = (s+1) (s+5/9) (s+8/9) (s+10/9) (s+7/9) (s+11/9) (s+13/9).
$$
Computing a Gr\"obner basis, one can see that the ideal in $D[s,T]$ generated by $\{P_1(s),P_2(s)$, $f, 1-T q(s)\}$ is the whole ring. From Theorem \ref{mainresult} (3), one deduces that $q(s)$ contains all the roots of $b_f(s)$. Using this approach we only have to check whether an ideal is the whole ring or not. Therefore any admissible monomial ordering can be chosen, hence the one, which is generically fast.
\end{ex}

 Given an arbitrary rational number $\alpha$, let us consider the ideal $I_\alpha\subseteq D_n[s]$ generated by the annihilator of $f^s$, the polynomial $f$ and $s+\alpha$. Theorem \ref{mainresult} (1) says that the equality $I_\alpha = D_n[s]$ holds generically (this is clarified in Corollary \ref{cor1} below). Hence the roots of the Bernstein-Sato polynomial are the rational numbers for which the condition $I_\alpha \neq D_n[s]$ is satisfied. This allows one to work out with parameters, that is over $\K(\alpha)\langle x,\partial_x \rangle[s]$, and find the corresponding complete set of special parameters. The latter procedure is algorithmic \cite{LZ07} and implemented in \textsc{Singular}. Note, that the set of candidates to obstructions, returned by the latter algorithm is in general bigger, than the set of real obstructions.




\begin{cor}\label{cor1}
Let $\{ P_1(s),\ldots,P_k(s) \}$ be a system of generators of the annihilator of $f^s$ in $D_n[s]$. The following conditions are equivalent:
\begin{enumerate}
\item $\alpha\in\mathbb{Q}_{>0}$ is a root of $b_f(-s)$.
\item $D_n[s] \langle P_1(s),\ldots,P_k(s),f,s+\alpha \rangle \neq D_n[s]$.
\item $D_n \langle P_1(-\alpha),\ldots,P_k(-\alpha),f \rangle \neq D_n$.
\end{enumerate}
Moreover, in such a case $D_n[s]\langle P_1(s),\ldots,P_k(s),f,s+\alpha \rangle \cap \K[s]
= \K[s] \langle s+\alpha\rangle$.
\end{cor}

\begin{proof}
Take $J = D_n[s]\langle P_1(s),\ldots,P_k(s),f,s+\alpha \rangle$ and $K=J\cap D_n
=\langle P_1(-\alpha),\ldots P_k(-\alpha),f \rangle$.
Since
$$
  J = D[s] \Longleftrightarrow J\cap \K[s] = \K[s] \Longleftrightarrow K = D_n,
$$
and $\gcd(b_f(s),s+\alpha) = 1$ if and only if $b_f(-\alpha)\neq 0$,
the result follows from applying Theorem \ref{mainresult} using $q(s)=s+\alpha$.
\end{proof}

Once we know a system of generators of the annihilator of $f^s$ in $D_n[s]$, the last corollary provides an algorithm for checking whether a given rational number is a root of the $b$-function of $f$, using Gr\"obner bases in the Weyl algebra.

\begin{algorithm}[H]
\caption{\textsc{checkRoot1} (checks whether $\alpha\in\mathbb{Q}_{>0}$ is a root of $b_f(-s)$)}
\label{checkRoot1}
\algsetup{indent=2em}
\begin{algorithmic}
\STATE Input 1: $\{P_1(s),\ldots,P_k(s)\}\subseteq D_n[s]$, a system of
generators of $\ann_{D_n[s]}(f^s)$;
\STATE Input 2: $f$, a polynomial in $R_n$;\quad $\alpha$, a number in $\mathbb{Q}_{>0}$;
\STATE Output: \TRUE, if $\alpha$ is a root of $b_f(-s)$;\quad \FALSE, otherwise;\bigskip
\STATE $K:= \langle P_1(-\alpha),\ldots,P_k(-\alpha),f\rangle$;
  \hfill $\triangleright K = J \cap D_n \subseteq D_n$
\STATE $G:=$ reduced Gr\"obner basis of $K$ w.r.t. ANY term ordering;
\RETURN ($G\neq \{1\}$);
\end{algorithmic}
\end{algorithm}

\subsection{Multiplicities}

Two approaches to deal with multiplicities are presented. We start with a natural generalization of Corollary \ref{cor1}.

\begin{cor}\label{corAlg2}
Let $m_\alpha$ be the multiplicity of $\alpha$ as a root of $b_f(-s)$ and let us consider the ideals $J_i = \ann_{D_n[s]}(f^s) + \langle f, (s+\alpha)^{i+1} \rangle \subseteq D_n[s],\ i=0,\ldots,n$.
The following conditions are equivalent:
\begin{enumerate}
\item $m_\alpha > i$.
\item $J_i\cap\K[s] = \langle (s+\alpha)^{i+1} \rangle$.
\item $(s+\alpha)^i \notin J_i$.
\end{enumerate}
Moreover if $D_n[s] \supsetneq J_0 \supsetneq J_1 \supsetneq \cdots \supsetneq J_{m-1} = J_m$, then $m_\alpha = m$. In particular, $m\leq n$ and $J_{m-1}=J_m = \cdots = J_n$.
\end{cor}

\begin{proof}
$1\Longleftrightarrow 2$. Since $m_\alpha>i$ if and only if
$\gcd(b_f(s),(s+\alpha)^{i+1}) = (s+\alpha)^{i+1}$, the equivalence follows by applying
Theorem \ref{mainresult} (1) using $q(s)=(s+\alpha)^{i+1}$.\\
$2\Longrightarrow 3$. If $(s+\alpha)^i \in J_i\cap \K[s]$, then clearly
$J_i\cap\K[s] \supsetneq \langle (s+\alpha)^{i+1} \rangle$.\\
$3\Longrightarrow 2$. Let $h(s)\in\K[s]$ be the monic generator of the ideal $J_i\cap\K[s]$.
Since $(s+\alpha)^{i+1}\in J_i\cap\K[s] = \langle h(s) \rangle$, there exists $j\leq i+1$
such that $h(s) = (s+\alpha)^j$. Suppose that $j\leq i$. Then,
$$
  (s+\alpha)^i = (s+\alpha)^{i-j} (s+\alpha)^j = (s+\alpha)^{i-j} h(s) \in J_i.
$$
That, however, contradicts 3 and thus $j=i+1$.

The rest of the assertion follows by applying the above result using
$i=m$ and $i=m-1$, since $(s+\alpha)^m \in J_m$ and $(s+\alpha)^{m-1}\notin J_{m-1}$
from the hypothesis.
\end{proof}

Again once we know a system of generators of the annihilator of $f^s$ in $D_n[s]$,
the last corollary provides an algorithm for checking whether a given
rational number is a root of the $b$-function of $f$ and for computing its
multiplicity, using Gr\"obner bases for differential operators.

\begin{algorithm}[H]
\caption{\textsc{checkRoot2} (computes the multiplicity of $\alpha\in\mathbb{Q}_{>0}$ as a root of $b_f(-s)$)}
\label{checkRoot2}
\algsetup{indent=2em}
\begin{algorithmic}
\STATE Input 1: $\{P_1(s),\ldots,P_k(s)\}\subseteq D_n[s]$, a system of generators of $\ann_{D_n[s]}(f^s)$;
\STATE Input 2: $f$, a polynomial in $R_n$;\quad $\alpha$, a number in $\mathbb{Q}_{>0}$;
\STATE Output: $m_\alpha$, the multiplicity of $\alpha$ as a root of $b_f(-s)$;\bigskip
\FOR {$i=0$ to $n$}
\STATE $J:= D_n[s]\cdot \langle P_1(s),\ldots,P_k(s), f, (s+\alpha)^{i+1}\rangle$;
  \hfill $\triangleright \ J_i$
\STATE $G:=$ Gr\"obner basis of $J$ w.r.t. ANY term ordering;
\STATE $r:=$ normal form of $(s+\alpha)^i$ with respect to $G$;\\
\IF {$r=0$}
\STATE $m_\alpha:=i$;
\hfill $\triangleright\ r=0 \Longrightarrow (s+\alpha)^i \in J_i$
\STATE {\bf break}
\hfill $\triangleright$\ leave the {\bf for} block
\ENDIF
\ENDFOR
\RETURN $m_\alpha$;
\end{algorithmic}
\end{algorithm}

\begin{proof} (of Algorithm \ref{checkRoot2}).\\
\textit{Termination:} The algorithm \textsc{checkRoot2} clearly terminates and
one only has to consider the loop from $0$ to $n$ because the multiplicity
of a root of $b_f(s)$ is at most 
$n$, see \cite{Saito94}.\\
\textit{Correctness:} Corollary \ref{corAlg2} implies the correctness of the method.
\end{proof}

\begin{remark}
There exists another version of {\tt checkRoot2} with just one step, due to the formula,
see Corollary \ref{corAlg2} above,
$$
  \big(\ann_{D_n[s]}(f^s) + D_n[s]\langle f, (s+\alpha)^n \rangle\big) \cap \K[s]
  = \langle (s+\alpha)^{m_\alpha} \rangle.
$$
However, this method only seems to be useful when the multiplicity is close to $n$,
otherwise {\tt checkRoot2} is more effective. The reason is that in general, the multiplicity
is far lower than the number of variables.
\end{remark}

This algorithm is much faster, than the computation of the whole Bernstein
polynomial via Gr\"obner bases, because no elimination ordering is needed
for computing a Gr\"obner basis of $J$. Also, the element $(s+\alpha)^{i+1}$,
added as a generator, seems to simplify tremendously such a computation. 
Actually, when $i=0$ it is possible to eliminate the variable $s$ in advance and
we can perform the whole computation in $D_n$, see Corollary \ref{cor1} (3) above.

Nevertheless, Algorithm \ref{checkRoot2} meets the problem to calculate on each step a
Gr\"obner basis $G_i$ for an ideal of the form $I+\langle (s+\alpha)^{i+1}\rangle$ and
the set $G_{i-1}$ is not used at all for such computation. A completely
new Gr\"obner basis has to be performed instead.
The classical idea of quotient and saturation are used to solve this obstruction.
In particular, the following result holds.

\begin{cor}\label{corAlg3}
Let $m_\alpha$ be the multiplicity of $\alpha$ as a root of $b_f(-s)$ and
let us consider the ideal $I = \ann_{D_n[s]}(f^s) + D_n[s] \langle f \rangle$. The
following conditions are equivalent:
\begin{enumerate}
\item $m_\alpha > i$.
\item $\big(\, I: (s+\alpha)^i\, \big) + D_n[s] \langle s+\alpha \rangle \neq D_n[s]$.
\item $\big(\, I: (s+\alpha)^i\, \big)|_{s=-\alpha} \neq D_n$.
\end{enumerate}
\end{cor}

\begin{proof}
Given $J\subseteq D_n[s]$ an ideal, we denote by $b_J(s)$
the monic generator of the ideal $J\cap \K[s]$. Then, from Theorem \ref{mainresult} (1),
condition 2 is satisfied if and only if $-\alpha$ is a root of $b_{I:(s+\alpha)^i}(s)$.
This univariate polynomial is nothing but $b_f(s)/\gcd(b_f(s),(s+\alpha)^i)$, due to
Theorem \ref{mainresult} (2). Now the claim follows from the obvious equivalence
$$
  m_\alpha > i \ \Longleftrightarrow\
  (s+\alpha) \,\Big|\, \frac{b_f(s)}{\gcd(b_f(s),(s+\alpha)^i)}.
$$
\end{proof}

Since $s+\alpha$ belongs to the center of $D_n[s]$, the ideal $I:(s+\alpha)^i$
can recursively be computed by the formulas
\begin{eqnarray*}
I:(s+\alpha) &=& (I\cap D_n[s] \langle s+\alpha\rangle)/(s+\alpha),\\
I:(s+\alpha)^i &=& (I:(s+\alpha)^{i-1}):(s+\alpha).
\end{eqnarray*}

The following is a sketch of another algorithm for computing multiplicities using
quotient ideals. The termination and correctness follow from the above corollary.

\begin{algorithm}[H]
\caption{\textsc{checkRoot3} (computes the multiplicity of $\alpha\in\mathbb{Q}_{>0}$ as a root of $b_f(-s)$)}
\label{checkRoot3}
\algsetup{indent=2em}
\begin{algorithmic}
\STATE Input 1: $\{P_1(s),\ldots,P_k(s)\}\subseteq D_n[s]$, a system of generators of $\ann_{D_n[s]}(f^s)$;
\STATE Input 2: $f$, a polynomial in $R_n$;\quad $\alpha$, a number in $\mathbb{Q}_{>0}$;
\STATE Output: $m_\alpha$, the multiplicity of $\alpha$ as a root of $b_f(-s)$;\bigskip
\STATE $m:=0$;\quad $I:= D_n[s] \langle P_1(s),\ldots,P_k(s),f \rangle$;
\quad $J:=I+D_n[s]\langle s+\alpha \rangle$;
\WHILE{$G \neq \{1\}$}
\STATE $m := m+1$;
\STATE $I:=I:(s+\alpha)$;
  \hfill $\triangleright \ I:(s+\alpha)^i$
\STATE $J:=I+D_n[s]\langle s+\alpha \rangle$; \quad (or $J:= I|_{s=-\alpha}$)
\STATE $G:=$ reduced Gr\"obner basis of $J$ w.r.t. ANY term ordering;
\ENDWHILE
\RETURN $m$;
\end{algorithmic}
\end{algorithm}

\begin{remark}
Several obvious modifications of the presented algorithms can be useful depending
on the context. Assume, for instance, that $q(s)$ is a known factor of the Bernstein-Sato
polynomial and one is interested in computing the rest of $b_f(s)$. Then
the ideal $I: q(s)$ contains such information. This easy observation can help us in
some special situations.
\end{remark}

\begin{remark}
\label{reducedB}
Define the \textbf{reduced \BS polynomial} of $f\in R_n$ to be $b'_f(s) = b_f(s)/(s+1)$. 
Recall, that the \textbf{Jacobian ideal} of $f$ is $J_f = \langle \tfrac{\d f}{\d x_1}, \ldots, \tfrac{\d f}{\d x_n} \rangle\subset \K[x]$. 
It is known, that taking $\langle f \rangle + J_f$ 
instead of $\langle f \rangle$ has the following consequence
\[
(\ann_{D[s]} f^s + \langle f, \tfrac{\d f}{\d x_1}, \ldots, \tfrac{\d f}{\d x_n} \rangle ) \cap \K[s] = \langle b'_f(s) \rangle = \langle \tfrac{b_f(s)}{s+1} \rangle.
\]
Hence, all the algorithms above can be modified to this setting, resulting in more effective computations. This is the way it should be done in the implementation. We decided, however, not to modify the description of algorithms in order to keep the exposition easier.
\end{remark}

\subsection{Local versus global $b$-functions}

Here we are interested in what kind of information one can obtain
from the global $b$-function for computing the local ones and
conversely. In order to avoid theoretical problems we will assume in
this paragraph that the ground field is $\C$.

Several algorithms to obtain the local $b$-function of a hypersurface $f$
have been known without any Gr\"obner bases computation but under strong
conditions on $f$. For instance, it was shown in \cite{Malgrange75} that the
minimal polynomial of $-\d_t t$ acting on some vector space of finite
dimension coincides with the reduced local Bernstein polynomial,
assuming that the singularity is isolated.

\begin{remark}
\label{Dloc}
Recall, that the singular locus of $V(f)$ is $V(\langle f,\tfrac{\d f}{\d x_1}, \ldots, \tfrac{\d f}{\d x_n} \rangle)$. One can define the \textbf{local $b$-function} or \textbf{local \BS polynomial}
as follows. Let $p\in\C^n$ be a point and $\mathfrak{m}_p = \langle \{ x_1 - p_1, \ldots, x_n - p_n \} \rangle \subset R_n$ the corresponding maximal ideal. Let $D_p$ be the local Weyl algebra at $p$, that is Weyl algebra with coefficients from $\C[x_1,\ldots,x_n]_{p}$ instead of $R_n=\C[x_1,\ldots,x_n]$. From the Bernstein's functional equation (\ref{bernstein}) it follows
that $\exists P(s)\in D[s], b_f(s)\in\K[s]$, such that $P(s)f\cdot f^s = b(s)\cdot f^s$ holds. Hence,
since over $\C[x_1,\ldots,x_n]_{p}$ we have got polynomial invertible, there exist
$\exists P_p(s)\in D_p[s], b_{f,p}(s)\in\K[s]$, such that $P_p(s)f\cdot f^s = b_{f,p}(s)\cdot f^s$ holds. We define local \BS polynomial to be the univariate monic polynomial $b_{f,p}(s)$ of the minimal degree, such that the above identity holds.
\end{remark}

\begin{theo}(Brian\c con-Maisonobe (unpublished), Mebkhout-Narv\'aez \cite{MN91})
\label{localBS}
Let $b_{f,p}(s)$ the local $b$-function of $f$ at the point $p\in\C^n$
and $b_f(s)$ the global one. Then 
$b_f(s) = \lcm_{p\in\C^n} b_{f,p}(s) = \lcm_{p\in\Sigma(f)} b_{f,p}(s)$.
\end{theo}

The previous Theorem can be very useful for computing the global $b$-function
using the local ones. Let us see an example.

\begin{ex}\label{excuplines}
Let $\mathcal{C}$ the curve in $\C^2$ given by the equation $f = (y^2-x^3)(3x-2y-1)(x+2y)$.
This curve has three isolated singular points $(0,0)$, $(1,1)$ and $(1/4,-1/8)$. The following is its real picture.

 \begin{figure}[H]
 \centering
 \includegraphics[width=.2\textwidth]{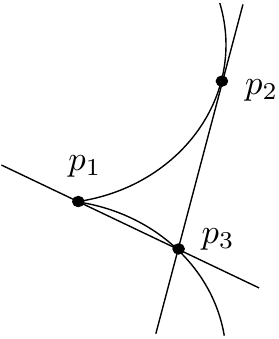}
 \caption{The cup $(2,3)$ with two lines.}
 \end{figure}

The library {\tt gmssing.lib} contains a procedure {\tt bernstein}, which computes
the local $b$-function at the origin.
Moving to the corresponding points we can also compute $b_{f,p_i}(s)$.
\begin{eqnarray*}
b_{f,p_1}(s) &=& (s+1)^2 (s+5/8) (s+7/8) (s+9/8) (s+11/8)\\
b_{f,p_2}(s) &=& (s+1)^2 (s+3/4) (s+5/4)\\
b_{f,p_3}(s) &=& (s+1)^2 (s+2/3) (s+4/3)
\end{eqnarray*}
From this information and using Theorem \ref{localBS}, the global $b$-function is
\[
(s+1)^2 (s+2/3) (s+5/8) (s+3/4) (s+7/8) (s+4/3)  (s+5/4)  (s+9/8) (s+11/8).
\]
\end{ex}

The computation of the global $b$-function with Theorem \ref{localBS}
is effective, when the singular locus consists of finitely many isolated singular points.
The {\sc Singular} library {\tt gmssing.lib} implemented
by M. Schulze \cite{Gmssinglib} and based on his work \cite{Schulze04} allows one to compute invariants related to the the Gauss-Manin system of an isolated hypersurface singularity.
In the non-isolated case the situation is more complicated, since no
Gauss-Manin connection exists. For computing the local $b$-function
in this case (which is important on its own) we suggest using the global 
$b$-function as an upper bound and a local version of the {\tt checkRoot} algorithm, 
see Section \ref{localcheckRoot} below.

In \cite{Nakayama09}, H.~Nakayama presented an algorithm for computing local $b$-functions. One step in his algorithm uses a bound for the multiplicity of a given rational root of the global
$b$-function. Then the algorithm checks if this multiplicity agrees with
the local one. This approach is very similar to our \texttt{checkRoot} algorithm.



\subsubsection{Localization of non-commutative rings}

We recall some properties of rings of fractions in non-commutative
setting. The reader is referred to \cite{GW04} and \cite{MR01} for further details.

\begin{defi}
Let $R$ be a ring and $S\subseteq R$ a multiplicatively closed set. A {\em left ring
of fractions} for $R$ with respect to $S$ is a ring homomorphism $\phi:R\to Q$
such that:
\begin{enumerate}
\item $\phi(s)$ is a unit of $Q$ for all $s\in S$.
\item Each element of $Q$ has the form for $\phi(s)^{-1}\phi(r)$ for some $r\in R$
and $s\in S$.
\item $\ker (\phi) = \{r\in R \mid s r = 0 \,\text{ for some\, $s\in S$}\}$.
\end{enumerate}
{\em Right rings of fractions} are defined analogously.
\end{defi}

\begin{theo}\label{localization}
There exists a left ring of fractions for $R$ with respect to $S$ if and only if $S$
is a left denominator set, that is, the following conditions hold:
\begin{itemize}
\item Left Ore condition: for each $r\in R$ and $s\in S$, there exist $r'\in R$ and
$s'\in S$ such that $s' r = r' s$, that is, $Sr\cap Rs \neq \emptyset$.
\item Left reversible: if $r s=0$ for some $r\in R$ and $s\in S$, then $\exists r' \in R$
such that $s r'=0$.
\end{itemize}
In such a case, the pair $(Q,\phi)$ is universal for homomorphisms $\varphi:R\to T$
such that $\varphi(S)$ consists of units of $T$ and therefore $Q$ is unique up to
unique isomorphism. Moreover, if $R$ also has a right ring of fractions $Q'$ with
respect to $S$ then $Q\simeq Q'$. 
\end{theo}

Because of the uniqueness, the left ring of fractions $Q$ (when it exists)
is often denoted by
$S^{-1}R$, and the natural map $\phi: R \to S^{-1}R$ is
called the {\em localization map}. To simplify notation the elements of
$S^{-1}R$ are denoted by $s^{-1} r$, even when $\ker\phi\neq 0$. Two quotients
$s_1^{-1} r_1$ and $s_2^{-1} r_2$ are equal if and only if there exist $s\in S$ and
$a\in R$ such that $a s_1 = s s_2$ and $a r_1 = s r_2$.
The localization for left (resp. right) modules can be generalized in the obvious way
and it is verified $S^{-1} M \cong S^{-1} R \otimes_R M$
(resp. $M S^{-1} \cong M\otimes_R R S^{-1}$).

Recall the following two classical results on localizations.

\begin{lemma}\label{lem1}
Let $R_1\stackrel{i}{\hookrightarrow} R_2$ be a ring extension and $S\subset R_1$ a multiplicatively closed set. Assume $S^{-1} R_1$ and $S^{-1} R_2$ exist and consider
the corresponding localization maps $\phi_1: R_1 \to S^{-1} R_1$ and
$\phi_2: R_2 \to S^{-1} R_2$. Let $j: S^{-1} R_1 \to S^{-1} R_2$ be the map induced
by $i$. Then $j$ is injective and for every left ideal $I\subseteq R_2$ one has
$$
  S^{-1} I \cap S^{-1} R_1 = S^{-1} (I\cap R_1).
$$
\end{lemma}

Note that in the previous lemma, $S^{-1} I$ is the extension of $I$ to $S^{-1} R_2$
while $S^{-1}(I\cap R_1)$ is the extension of $I\cap R_1$ to $S^{-1} R_1$.

\begin{lemma}\label{lem2}
Let $R$ be a ring, $S\subseteq R$ a multiplicatively closed set and $I\subseteq R$ a left ideal.
Assume $S^{-1}R$ exists. Then $S^{-1}I$ is not the whole ring $S^{-1}R$
if and only if $I\cap S = \emptyset$.
\end{lemma}

\begin{ex}
Let $R = D$ be the classical $n$-Weyl algebra and $S=\K[{\bf x}]\setminus \mathfrak{m}_p$,
where $p\in \K^n$ is an arbitrary point, cf. Remark \ref{Dloc}. 
Then $S$ is a left and right denominator set as in the statement of Theorem \ref{localization}, and the localization
$(\K[{\bf x}]\setminus \mathfrak{m}_p)^{-1} D$ is naturally isomorphic to
$D_p$. Analogous construction also holds for the extension $D[s] = \K[s]\otimes_\K D$.
\end{ex}

\subsubsection{Local version of the {\tt checkRoot1} algorithm}\label{localcheckRoot}

Theorem \ref{mainresult} is general enough to apply also for checking rational roots of
local Bernstein-Sato polynomials.
To simplify the exposition, we concentrate our attention on the local version of
{\tt checkRoot1} algorithm. See Section \ref{stratum} for other generalizations.

Let $f\in \C[x_1,\ldots,x_n]$ be a polynomial, $p\in \C^n$ and $\alpha\in \mathbb{Q}$.
Then the first part of Theorem \ref{mainresult}, see also Corollary \ref{cor1}, tells us
that $(s+\alpha)$ is a factor of the local $b$-function at $p$ if and only if the left ideal
\begin{equation}\label{idealOne}
  \ann_{D_p[s]}(f^s) + D_p[s] \langle f, s+\alpha \rangle
\end{equation}
is not the whole ring $D_p[s]$. From Lemma \ref{lem1} using $R_1 = D[s]$,
$R_2 = D\langle t, \partial t \rangle := D \otimes_\K \K\langle t, \partial t \mid \d t \cdot t = t \cdot \d t +1 \rangle$, $S=\C[{\bf x}]\setminus \mathfrak{m}_p$
and $I= \ann_{D\langle t, \partial t \rangle}(f^s) = I_f$ the Malgrange ideal
associated with $f$, one obtains
$$
  \ann_{D_p[s]}(f^s) = D_p[s] \ann_{D[s]}(f^s).
$$

\begin{prop}\label{proplocalCR}
Let $\{P_1(s),\ldots,P_k(s)\}$ be a system of generator of $\ann_{D[s]}(f^s)$
and consider the ideal $I = D[s] \langle P_1(s),\ldots,P_k(s),f,s+\alpha \rangle$.
Then we have
$$
  (s+\alpha)\, |\, b_{f,p}(s) \ \Longleftrightarrow \
  p\in V(I\cap \C[{\bf x}]).
$$
\end{prop}

\begin{proof}
From the above discussion, $D_p[s] I$ equals the ideal of the equation (\ref{idealOne})
and thus $(s+\alpha)$ is a factor of $b_{f,p}(s)$ iff $D_p[s] I\neq D_p[s]$.
Now, by Lemma \ref{lem2} using $R=D[s]$ and $S=\C[{\bf x}]\setminus \mathfrak{m}_p$
$$
  D_p[s] I \neq D_p[s] \ \Longleftrightarrow \
  I\cap (\C[{\bf x}]\setminus \mathfrak{m}_p) = \emptyset \ \Longleftrightarrow \
  I\cap \C[{\bf x}] \subseteq \mathfrak{m}_p
$$
and the claim follows.
\end{proof}

There are several ways to check whether an ideal $I\subseteq D_p[s]$ is proper or not.
However, it is an open problem to decide which one is more efficient. Mora division
and standard bases techniques seem to be more suitable in this case,
since otherwise a (global) elimination ordering is needed. On the other hand,
using this approach, such orderings are unavoidable for obtaining the stratification
associated with local $b$-functions, see Section \ref{stratum} where several examples
are showed.

\subsection{$b$-functions with respect to weights and {\tt checkRoot}}

The $b$-function associated with a holonomic ideal with respect to a weight is presented.
We refer \cite{SST00} for the details.
Let $0 \neq w \in \R^n_{\geq 0}$ and consider the $V$-filtration with respect to $w$, $\left\{ V_m \mid m \in \Z \right\} = V$ on $D$   where $V_m$ is spanned by $\left\{ x^{\alpha} \d^{\beta} \mid -w \alpha + w \beta \leq m \right\}$ over $\K$.
That is, $x_i$ and $\d_i$ get weights $-w_i$ and $w_i$ respectively. Note, that with respect to such weights the relation $\d_i x_i = x_i \d_i + 1$ is homogeneous of degree $0$.
The associated graded ring $\gr^V(D) = \bigoplus_{m \in \Z} V_m / V_{m-1}$
is isomorphic to $D$, which allows us to identify them.

For a non-zero operator
\[
  P = \sum_{\alpha, \beta \in \N^n} a_{\alpha \beta} x^{\alpha} \d^{\beta} \in D,
\]
the maximum $\max_{\alpha, \beta} \{ -w \alpha + w \beta \mid c_{\alpha \beta} \neq 0\} \in \R$ is denoted by $\ord^V(P)$ and the principal symbol of $P$ is the $V$-homogeneous
operator given by
\[
\sigma^V(P) := \sum_{-w \alpha + w \beta = \ord_V(P)}
a_{\alpha \beta} x^{\alpha} \d^{\beta}.
\]
Additionally, for a given ideal $I\subseteq D$, the associated graded ideal
is defined as the vector space spanned by all its principal symbols, that is,
$\gr^V(I) := \K \cdot \{ \sigma^V(P) \mid P \in I \}$.

Sometimes, the principal symbol (resp. associated graded ideal) is called the initial form
(resp. initial ideal) and it is denoted by $\ini_{(-w,w)}(P)$ (resp. $\ini_{(-w,w)}(I)$).

\begin{defi}
Let $I\subset D$ be a holonomic ideal. Consider $0 \neq w \in \R^n_{\geq 0}$ and
$s := \sum_{i=1}^n w_i x_i \d_i$.
Then $\gr^V (I) \cap \K[s]\neq 0$ is a principal ideal in $\K[s]$.
Its monic generator is called the \emph{global $b$-function of $I$ with respect to the weight $w$}.
\end{defi}

Although Theorem \ref{mainresult} can not be applied in this setting, since
$s = \sum_i w_i x_i \partial_i$ does not belong to the center of the algebra,
a similar result still holds, due to the properties of the $V$-filtration,
see Proposition \ref{propInitial} below. Also Corollaries \ref{cor1}, \ref{corAlg2}
and \ref{corAlg3} can be established using initial parts instead of annihilators.

\begin{prop}\label{propInitial}
$\big(\, \gr^V(I) + \gr^V(D) \langle q(s) \rangle\, \big)\cap \K[s] = \gr^V(I)\cap\K[s] +
\K[s] \langle q(s) \rangle$.
\end{prop}

\begin{proof}
Actually it is an easy consequence of being treated with $V$-homogeneous ideals. Consider
$h(s) = Q + R\cdot q(s)$, where $Q\in \gr^V(I)$ and $R\in \gr^V(D)$. Taking $V$-homogeneous parts in the above expression, one finds $Q_0 \in \gr^V(I)$ and $R_0\in \gr^V(D)$ of
degree $0$ such that
$
  h(s) = Q_0 + R_0\cdot q(s).
$
Now, since $q(s)$ commutes with $Q_0$, one can proceed as in the proof of Theorem
\ref{mainresult} (1).
\end{proof}

Many algorithms in the realm of $D$-modules are based on the computation of
such $b$-functions. For some applications like integration and restriction,
only the maximal and the minimal integral roots have to be computed.

However the above proposition can not be used to find the set of all integral roots, since neither upper nor lower bound is known in advance. For instance, N.~Takayama used the following easy example to show the general unboundness: $I=\langle x\d_1 + k\rangle$, $k\in\mathbb{Z}$ is $D_1$-holonomic and $\ini_{(-1,1)}(I)\cap \C[s] = \langle s+k\rangle$ with $s=t\d_t$.

\section{Computing $b$-functions via Upper Bounds}

As different possible ways to find upper bounds, we present embedded resolutions,
topologically equivalent singularities and A'Campo's formula.
Depending on the context local or global version of our algorithm is used.

\subsection{Embedded resolutions}
\label{upperBound}

In this part of the paper we will work over
the field $\C$ of the complex numbers. However, in actual computation
we can assume that the ground field is generated by a finite number of
(algebraic or transcendental) elements over the field $\mathbb{Q}$ of
the rational numbers and that the algebraic relations among these
elements are specified. 

\begin{defi}
Let $h:Y \to \C^n$ be a proper birational morphism. We say that $h$
is a {\em global embedded resolution} of the hypersurface
defined by a polynomial $f\in \C[x]$, $X=V(f)$, if the following conditions are satisfied:
\begin{enumerate}
\item $Y$ is a non-singular variety.
\item $h: Y\setminus h^{-1}(X) \to \C^n\setminus X$ is an isomorphism.
\item $h^{-1}(X)$ is a normal crossing divisor.
\end{enumerate}
\end{defi}

Since $h^{-1}(X)$ is a normal crossing divisor, the morphism
$F=f\circ h : Y \to \C$ is locally given by a monomial.
Hence, we can define the $b$-function of $F$ as the least common
multiple of the local ones. If $F$ is locally given by the monomial
$x^\alpha = x_1^{\alpha_1}\cdots x_n^{\alpha_n}$ at the point $p$,
then
$$
  b_{F,p}(s) = \prod_{i=1}^{\alpha_1} \left(s+\frac{i}{\alpha_1}\right)
  \cdots \prod_{i=1}^{\alpha_n} \left(s+\frac{i}{\alpha_n}\right) =
  \prod_{1\leq i_j\leq \alpha_j}\ \prod_{1\leq k\leq n}
  \left(s+\frac{i_k}{s_k}\right).
$$

The following is the global version of the classical result by Kashiwara
\cite{Kashiwara76/77}. The upper bound statement is due to
Varchenko (\cite{Varchenko81}) and Saito (\cite{Saito93, Saito94}).

\begin{theo}
\label{VS}
For $f\in R_n$, there exists an integer $k$ such that $b_f(s)$ is a divisor of the product $b_F(s) b_F(s+1) \cdots b_F(s+k)$. Moreover $0 \leq k \leq n-1$. 
\end{theo}

\begin{proof}
Since $h$ is a global embedded resolution of $X=V(f)$, $h$ induces
a local embedded resolution of the germ $(X,p)$ at every point
$p\in X$. Now, the 
existence of $k\geq 0$ with the divisibility property follows from the theorem by Kashiwara \cite{Kashiwara76/77} and from the fact that the global $b$-function is the least common multiple of
the local ones, see Theorem \ref{localBS}. The proof for the upper bound can be found in the
references above.
\end{proof}

This theorem allows one to find upper bounds also for the global case.
Let us see an example to show how one can apply the algorithm
{\tt checkRoot} in order to compute the $b$-function.

%

\begin{ex}\label{exCheckRoot}
Let $f=(x z+y)(x^4+y^5+x y^4)\in\mathbb{Q}[x,y,z]$ and
$B_1(s) = b_{x^5}(s) b_{y^{18}}(s) b_{z^{24}}(s)$.
Since every root of $b_f(-s)$ belongs to the real
interval $(0,3)$, see Theorem \ref{VS},
computing an embedded resolution of the singularity
and using Kashiwara's result \cite{Kashiwara76/77},
we obtain that $B(s)=B_1(s)B_1(s+1)B_1(s+2)$ is an upper
bound for $b_f(s)$. Once we know a system of generators
of the $\ann_{D_n[s]} f^s$, checking whether
each root of the upper bound is a root of the Bernstein-Sato
polynomial was easy. It took less than 5 seconds except
for those ones which appear in the table below.
We also observe that when a candidate is not a root indeed, 
the computation is very fast.
To the best of our knowledge,
this example (first appeared in \cite{CU05}) is intractable
by any computer algebra system.
\begin{equation*}
\begin{split}
  b_f(s) =\ &(s+1)^2 (s+17/24) (s+5/4) (s+11/24) (s+5/8) (s+31/24) (s+13/24)\\
  & (s+13/12) (s+7/12) (s+23/24) (s+5/12) (s+3/8) (s+11/12) (s+9/8)\\
  & (s+7/8) (s+19/24) (s+3/4) (s+29/24) (s+25/24)
\end{split}
\end{equation*}

The running time is given in the format \textit{minutes:seconds}.

\begin{center}
\hspace{0.5cm}\begin{tabular}{|c|c|c|}
\hline
Root of $B(-s)$ & Running time & Root of $b_f(-s)$ ?\\
\hline
$5/4$ & 29:16 & Yes\\
\hline
$31/24$ & 26:16 & Yes\\
\hline
$29/24$ & 7:51 & Yes\\
\hline
$9/8$ & 0:35 & Yes\\
\hline
\end{tabular}
\end{center}
\end{ex}

Let us give a brief indication for computing a global embedded resolution of $f$.
Consider $\pi: \widehat{\C}^3 \to \C^3$, the blow-up of $\C^3$ with center
in $Z=\{ x=y=0 \}$. Denote $V_1 = V(x z+y)$ and $V_2=V(x^4 + y^5 + x y^4)$ the
two components of $V(f)$ and $\widehat{V}_1$, $\widehat{V}_2$ their corresponding
strict transforms. The exceptional divisor $E$ has multiplicity $5$ and 
$\widehat{V}_1$ and $\widehat{V}_2$ do not meet. Moreover $\widehat{V}_1$ and
$E$ intersect transversally. The local equation of $\widehat{V}_2 \cup E$ is
given by the polynomial $y^5 (x^4 + y + x y)$. Now, one can proceed as in the
case of plane curves, since the local equation involves just two variables.
Finally, we obtain seven divisors with normal crossings, see Figure
\ref{embRes}. This method can also be applied to the family
$(x z+y) g(x,y)$ under some extra conditions on $g(x,y)$.

\begin{figure}[h t]\label{embRes}
\centering
\includegraphics{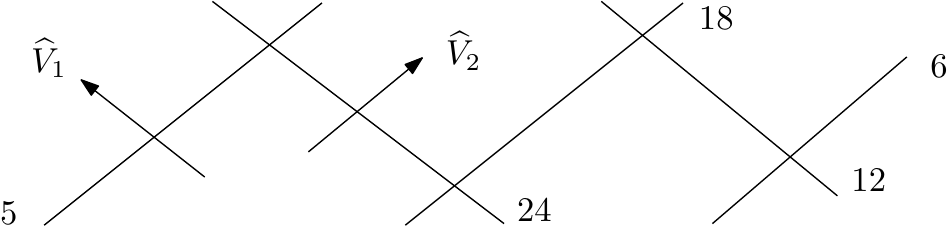}
\caption{Embedded resolution of $V((x z+y)(x^4+y^5+x y^4))$}
\end{figure}

\begin{remark}
To the best of our knowledge, resolution of singularities has 
never been used before for computing
Bernstein-Sato polynomials in an algorithmic way.
Recall that an embedded resolution can be computed algorithmically
in any dimension and for any affine algebraic variety \cite{BEV05}.
\end{remark}

One can find upper bounds for the case of hyperplane arrangements by computing an embedded resolution. This allows one among other to test formulas for \BS polynomials of non-generic arrangements. A formula for the \BS polynomial of  a generic hyperplane arrangement was given by Walther in \cite{Walther05}.


\subsection{Topologically equivalent singularities}
\label{topEquiSing}

Let $f, g$ be two topologically equivalent singularities and assume that
$b_f(s)$ is known. Since the set $E_f = \{ e^{2\pi i \alpha} \mid b_{f,0}(\alpha)=0 \}$
is a topological invariant of the singularity $\{ f=0 \}$ at the origin \cite{Malgrange75, Malgrange83} and every root belongs to $(-n,0)$ (Theorem \ref{VS}), one can find an
upper bound for $b_g(s)$ from the roots of $b_f(s)$ and use our algorithms for computing $b_g(s)$. The upper bound is constructed as $\prod_{\beta\in E} (s-\beta)$, where
$E = \{ \alpha + k \mid \alpha\in E_f, k \in \Z, \alpha+k \in (-n,0)\}$.



In general it is complicated to check, whether two singularities are equivalent. However, there are some special families for which this can be done. This is the case of quasi-ordinary singularities, see e.g.~\cite{Lipman88}. Let us see an example of a non-isolated one.

\begin{ex}
Let $f =  z^4 + x^6 y^5$ and $g = f + x^5 y^4 z$. Since the corresponding discriminants with respect to $z$ are normal crossing divisors, the associated germs at the origin define quasi-ordinary singularities. Moreover the characteristic exponents are in both cases the same and hence they are topologically equivalent, see e.g.~\cite{Lipman88}.

The Bernstein-Sato polynomial of $f$ at the origin has $27$ roots, all of them with multiplicity one except for $\alpha = -1$ which has multiplicity two. Here is the list in positive format.
\begin{small}
$$
  1, \frac{5}{6}, \frac{9}{10}, \frac{4}{3}, \frac{13}{10}, \frac{2}{3}, \frac{3}{4},
  \frac{19}{20}, \frac{5}{12}, \frac{11}{10}, \frac{17}{12}, \frac{17}{20}, \frac{11}{12},
  \frac{7}{10}, \fbox{$\displaystyle\frac{19}{12}$}, \frac{13}{20}, \frac{27}{20},
  \frac{7}{6}, \frac{21}{20}, \frac{9}{20}, \frac{13}{12}, \frac{5}{4}, {\bf \frac{3}{2}},
  \frac{7}{12}, {\bf \frac{31}{20}}, \fbox{$\displaystyle\frac{7}{4}$}, \frac{23}{20}
$$
\end{small}

The exponential of the above set has $24$ elements. Each of them gives three candidates for~$b_{g,0}(-s)$ except for $-\alpha = 1$ which gives just two. For instance $-\alpha = 1/2$ gives the following three possible roots.
$$
\frac{1}{2} \rightarrow \Big\{ \frac{1}{2}, \frac{3}{2}, \frac{5}{2} \Big\}
$$
There are $71$ possible roots in total. Note that using this approach we do not have any information about the multiplicities. Finally one obtains the roots for $b_{g,0}(-s)$.
\begin{small}
$$
  1, \frac{5}{6}, \frac{9}{10}, \frac{4}{3}, \frac{13}{10}, \frac{2}{3}, \frac{3}{4},
  \frac{19}{20}, \frac{5}{12}, \frac{11}{10}, \frac{17}{12}, \frac{17}{20}, \frac{11}{12},
  \frac{7}{10},\quad \frac{13}{20}, \frac{27}{20}, \frac{7}{6}, \frac{21}{20},
  \frac{9}{20}, \frac{13}{12}, \frac{5}{4}, {\bf \frac{1}{2}}, \frac{7}{12},
  {\bf \frac{11}{20}}, \quad \frac{23}{20}
$$
\end{small}

Observe that the Bernstein polynomials are very similar. The roots of $b_{f,0}(-s)$ marked with a box have disappeared in $b_{g,0}(-s)$ and the ones in bold $3/2$, $31/20$ have become $1/2$,~$11/20$. In the table we put the information on timings in minutes:seconds format. Here, \texttt{bfct} (computing Bernstein-Sato polynomial) and \texttt{Sannfs} (computing $\ann_{D[s]} (f^s)$) stand for corresponding procedures from the library \texttt{dmod.lib}. Also there are the minimal and the maximal time, spent for checking single roots.
\begin{center}
\begin{tabular}{|c|c|c|c|c||c|}
\hline
$\texttt{bfct}(f)$ & $\texttt{Sannfs}(g)$ & check all roots & $-\alpha = 27/20$ & $-\alpha = 1/2$ & $\texttt{bfct}(g)$\\
\hline
0:13 & 0:45 & 2:37 & 0:04 & 0:02 & 33:03\\
\hline
\end{tabular}
\end{center}

\end{ex}

\subsection{A'Campo's formula}
\label{acampo}

The Jordan form of the local Picard-Lefschetz monodromy of superisolated surface singularities was calculated by Artal-Bartolo in \cite{Artal94}. The main step in this computation was to present explicitly an embedded resolution for this family and study
the mixed Hodge structure of the Milnor fibration.

Since every root of the Bernstein-Sato polynomial belongs to the interval $(-n,0)$ (Theorem \ref{VS}) and the characteristic polynomial is a topological invariant, using the results by Malgrange \cite{Malgrange75, Malgrange83}, one can
eventually provide an upper bound for the $b$-function. Let us see an example
that was not feasible even with the powerful specialized implementation by Schulze
\cite{Gmssinglib}.

\begin{ex}
Let $V$ be the superisolated singularity defined by $f = z^6 + (x^4 z + y^5 + x y^4)$.
The characteristic polynomial is
$$
  \Delta(t) = \frac{(t^5-1)(t^6-1)(t^{120}-1)}{(t-1)(t^{30}-1)(t^{24}-1)}.
$$

This polynomial has $76$ different roots and thus we know in advance that
the Bernstein-Sato polynomial (resp. the reduced one) has at least $77$ (resp. $76$) different
roots. Using the above results in $230$ possible candidates. Only $77$ of them are roots of the $b$-function indeed, all of them with multiplicity one. The total running time was $1$ hour and $32$ minutes.
$$
\begin{array}{c}
1, \frac{27}{40}, \frac{101}{120}, \frac{41}{60}, \frac{17}{20}, \frac{83}{120},
\frac{103}{120}, \frac{43}{60}, \frac{53}{60}, \frac{29}{40}, \frac{107}{120},
\frac{23}{40}, \frac{89}{120}, \frac{109}{120}, \frac{71}{120}, \frac{91}{120},
\frac{37}{40}, \frac{73}{120}, \frac{31}{40},
\\[0.2cm]
\frac{113}{120}, \frac{37}{60},\frac{47}{60},\frac{19}{20}, \frac{77}{120}, \frac{97}{120}, \frac{39}{40}, \frac{13}{20}, \frac{49}{60},
\frac{59}{60}, \frac{79}{120}, \frac{33}{40}, \frac{119}{120}, \frac{3}{5},
\frac{4}{5}, \frac{121}{120}, \frac{47}{40}, \frac{161}{120}, \frac{181}{120},
\frac{61}{60},
\\[0.2cm]
\frac{71}{60}, \frac{27}{20}, \frac{91}{60} \frac{41}{40},\frac{143}{120}, \frac{163}{120}, \frac{61}{40}, \frac{21}{20}, \frac{73}{60}, \frac{83}{60}, \frac{31}{20}, \frac{127}{120}, \frac{49}{40},
\frac{167}{120}, \frac{187}{120}, \frac{43}{40}, \frac{149}{120}, \frac{169}{120},
\frac{131}{120},
\\[0.2cm]
\frac{151}{120}, \frac{57}{40}, \frac{133}{120}, \frac{51}{40}, \frac{173}{120}, \frac{67}{60}, \frac{77}{60}, \frac{29}{20}, \frac{137}{120}, \frac{157}{120}, \frac{59}{40}, \frac{23}{20}, \frac{79}{60}, \frac{89}{60}, \frac{139}{120}, \frac{53}{40},\frac{179}{120}, \frac{6}{5}, \frac{7}{5}
\end{array}
$$
\end{ex}

\begin{remark}
Spectral numbers are defined using the semi-simple part of the action of the monodromy
on the mixed Hodge structure on the cohomology of the Milnor fiber \cite{Steenbrink77},
\cite{Varchenko81}. In \cite[Th. 3.3]{GH07}, \cite[Th. 0.7]{Saito93} it is proved,
that some roots of the Bernstein-Sato polynomial of a germ with an isolated critical
point at the origin, can be obtained from the knowledge of the spectral numbers of
the germ. 
Since spectral numbers do not change under $\mu$-constant deformations, this also
gives a set of common roots of the Bernstein-Sato polynomials, associated with the members
of a $\mu$-constant deformation of a germ. Therefore, they provide a lower bound for
$b_f(s)$, as well as an upper bound.
\end{remark}


\section{Integral Roots of $b$-functions}

For several applications only integral roots of the $b$-function are needed,
e.g.~\cite{SST00}. We present here problems related to the so-called Logarithmic Comparison
Theorem and Intersection Homology $D$-module. Depending on the context local or global
version of our algorithm is used.

\subsection{Upper bounds from different ideals}
\label{diffIdeals}

Consider a left ideal $I\subseteq \ann_{D[s]} f^s$. Then $I + \langle f \rangle \subseteq \ann_{D[s]} f^s + \langle f \rangle \subsetneq D[s]$, that is the former is a proper ideal. 
Then define the \textbf{relative $b$-polynomial} $b^{I}_f(s)\in\K[s]$ to be the monic generator of $\big(I + \langle f \rangle\big) \cap \K[s]$, then $b_f(s) \mid b^{I}_f(s)$. Note, that quite often $b^{I}_f(s)=0$. But if $b^{I}_f(s)\not=0$, it
gives us an upper bound for $b_f(s)$. In particular, one can take $I$, giving rise to
a holonomic $D[s]$-module, that is $\mbox{GK.dim} D[s]/I = \mbox{GK.dim} D[s]/\ann_{D[s]} f^s = n+1$. 

Since $(s+1) \mid b_f(s)\mid b^{I}_f(s)$, one can consider the \textbf{reduced relative $b$-polynomial} $\widetilde{b^{I}_f(s)}\in\K[s]$ to be the monic generator of $\big(I + \langle f, \frac{\d f}{\d x_1}, \ldots, \frac{\d f}{\d x_n}\rangle\big) \cap \K[s]$.


A prominent example of $I$ as above is the logarithmic annihilator. 
Let $I = \ann^{(1)}_{D[s]}(f^s)$ be the ideal in $D[s]$ generated by the operators $P(s)\in \ann_{D[s]}$ of total degree at most one in $\d_i$.
Let us define $b_f^{(1)}(s) := b^{I}(f^s)_f(s) = \big(\ann^{(1)}_{D[s]}(f^s) + D[s] \langle f \rangle \,\big) \cap \K[s]$. The reduced $\widetilde{b_f^{(1)}}(s)$ is useful as well.


\subsection{Minimal integral root of $b_f(s)$ and the logarithmic comparison problem}
\label{minRoot}

Since every root of $b_f(s)$ belongs to the real interval $(-n,0)$, 
integral roots are bounded and therefore the whole
\BS 
polynomial is not needed. Let us see an example that could not be
treated before with the classical methods.

\begin{ex}\label{bigEx}
Let $A$ be the matrix given by
$$
A = \left(\begin{array}{c c c c}
x_1 & x_2 & x_3 & x_4\\
x_5 & x_6 & x_7 & x_8\\
x_9 & x_{10} & x_{11} & x_{12}
\end{array}\right).
$$
Let us denote by $\Delta_i$, $i=1,2,3,4$, the determinant of the minor
resulting from deleting the $i$-th column of A, and consider
$f= \Delta_1 \Delta_2 \Delta_3 \Delta_4$. The polynomial $f$
defines a non-isolated hypersurface in $\C^{12}$. 
Following Theorem \ref{VS}, 
the set of all possible integral roots of $b_f(-s)$ is $\{11,10,9,8,7,6,5,4,3,2,1\}$.





Using the algorithm {\tt checkRoot} with the logarithmic annihilator, see Section \ref{diffIdeals} above, instead of the classical one, we have proved that the only integral root of $b_f^{(1)}(s)$ in $(-12,0)$ is $-1$.
Hence $-1$ is the minimal integral root of $b_f(s)$.
The following is the timing information of the whole procedure.
Of course, $-1$ is always the root, but it is interesting to compare the timings of confirming this fact.
\begin{center}
\begin{tabular}{|c|c|c|}
\hline
Possible integral roots & 1 & 2 ... 11\\
\hline
Root of $b_f^{(1)}(s)$ ? & Yes & No\\
\hline
Running time & 3:01:12 & $\sim$ 7:50\\
\hline
\end{tabular}
\end{center}
\end{ex}

This example was suggested by F.~Castro-Jim\'enez and
J.-M.~Ucha for testing the Logarithmic Comparison Theorem, see e.~g. \cite{Torrelli07}.

The use of logarithmic annihilator allowed to reduce the computation time.
However, for $f$ from this example it is known, that $\ann_{D_n[s]}(f^s) = \ann^{(1)}_{D_n[s]}(f^s)$ and this fact together with some homogeneous properties were used to compute other roots of $b_f(s)$, see Example~\ref{continueEx} below.

\subsubsection{Quasi-homogeneous polynomials}
Assume $F\in R_n$ is a $w$-quasi-homogeneous polynomial with $w_i\neq 0$, that is,
there are numbers $w_1,\ldots,w_n$ such that with $\xi = \sum_{i=1}^n w_i x_i \partial_i$ one has $F = \xi (F)$.
Take $c\in \K^{*}$ and let us denote $f = F_{|x_k=1}$ for some fixed $k$.
We are interested in studying the relationship between the Bernstein-Sato polynomials
of $f$ and $F$. The result has been obtained working out directly with the functional
equation.


\begin{prop}\label{propHomog}
Let $F\in R_n$ be a quasi-homogeneous polynomial with respect to the weight vector
$w = (w_1,\ldots,w_n)$. Assume $w_k\neq 0$ for some
$k\in\{1,\ldots,n\}$ and define $f$ to be the polynomial, resulting from making the substitution $x_k=c\in \K^{*}$ in $F$.
Then $b_f(s)$ divides~$b_F(s)$.
\end{prop}


\begin{proof}
Consider the $V$-filtration on $D_n$ given by the variable $x_k$. Let $P(s)\in D_n[s]$ a differential operator satisfying the
functional equation for $F$. There exists $d\geq 0$ such that $x_k^d P(s) \in \sum_{i\geq 0} x_k^i \cdot V_0$.
From the homogeneity of $F$ one can deduce that
$$
  x_k \partial_k \bullet F^{s+1} = \frac{1}{w_k} \Big(s+1 - \sum_{i\neq k} w_i x_i \partial_i\Big) \bullet F^{s+1}.
$$

Let $D'$ be the $(n-1)$-th Weyl algebra in the variables $x_1,\ldots,\hat{x_k},\ldots,x_n$. Thus $V_0 = D'[x_k \partial_k]$
and $x_k^d P(s) \cdot F^{s+1}$ can be written in the form $Q(s)\cdot F^{s+1}$ where the operator $\partial_k$ does not
appear in $Q(s)\in D_n[s]$. The functional equation for $F$ has been converted in the following one:
$$
  x_k^d P(s) \bullet F^{s+1} = Q(s) \bullet F^{s+1} = x_k^d\, b_F(s) \bullet F^s
$$

Now the substitution $x_k=c\in \K^{*}$ can be made and the claim follows.
\end{proof}

\begin{ex}
The Bernstein-Sato polynomials of $F=x^2 z+y^3$ and $f=F_{|z=1} = x^2+y^3$ are
\begin{eqnarray*}
b_F (s) &=& \underbrace{(s+1) \left(s+\frac{5}{6}\right) \left(s+\frac{7}{6}\right)}
\left(s+\frac{4}{3}\right) \left(s+\frac{5}{3}\right).\\
&& \hspace{1.8cm} b_f(s)
\end{eqnarray*}
From the result of Kashiwara \cite{Kashiwara76/77} one can see, blowing up the origin of $F$, that the last two factors are related to the $b$-function of $\{z^3=0\}$. This is a general fact.
\end{ex}

\begin{ex}\label{continueEx}
Now, we continue Example \ref{bigEx}.
Let $g$ be the polynomial, resulting from $f$ by substituting $x_1$, $x_2$, $x_3$, $x_4$, $x_5$,
$x_9$ with $1$. Using Proposition \ref{propHomog} several times, one can easily see that
$b_g(s)$ divides $b_f(s)$. Finally, the {\tt checkRoot} algorithm is used to obtain that
$$
  (s+1)^4 (s+1/2) (s+3/2) (s+3/4) (s+5/4)
$$
is a factor of $b_g(s)$ and therefore a factor of $b_f(s)$.
\end{ex}

\subsection{Intersection homology $\mathcal{D}$-module}
\label{interHomology}

In this part of the paper we introduce some new notation.
We refer to \cite{Torrelli09} for further details.
Let $X$ be a complex analytic manifold of dimension $n\geq 2$,
$\mathcal{O}_X$ the sheaf of holomorphic function on $X$ and
$\mathcal{D}_X$ the sheaf of differential operators with
holomorphic coefficients. At a point $x\in X$, we identify the
stalks $\mathcal{O}_{X,x}$ with the ring $\mathcal{O} =
\mathbb{C}\{x_1,\ldots,x_n\}$ of converging power series
and $\mathcal{D}_{X,x}$ with
$\mathcal{D}=\mathcal{O} \langle \partial_1,\ldots,\partial_n \rangle$.

Given a closed subspace $Y\subset X$ of pure codimension $p\geq 1$,
we denote by $H_{[Y]}^p(\mathcal{O}_X)$ the sheaf of local
algebraic cohomology with support in $Y$. Let $\mathcal{L}(Y,X)
\subset H_{[Y]}^p(\mathcal{O}_X)$ be the intersection homology
$\mathcal{D}_X$-Module of Brylinski-Kashiwara. This is the smallest
$\mathcal{D}_X$-submodule of $H_{[Y]}^p(\mathcal{O}_X)$ which
coincides with $H_{[Y]}^p(\mathcal{O}_X)$ at the generic points
of $Y$.

A natural problem is to characterize the subspaces $Y$
such that $\mathcal{L}(Y,X)$ coincides with $H_{[Y]}^p(\mathcal{O}_X)$.
Indeed, from the Riemann-Hilbert correspondence of Kashiwara-Mebkhout, 
the regular holonomic $D_X$-module $H_{[Y]}^p(\mathcal{O}_X)$ corresponds to the
perverse sheaf ${\bf C}_Y [p]$, while $\mathcal{L}(Y, X)$ corresponds to the
intersection complex $I C_Y^\bullet$. This way, the condition
$\mathcal{L}(Y, X) = H_{[Y]}^p(\mathcal{O}_X)$ is equivalent to the following one:
the real link of $Y$ at a point $x\in Y$ is a rational homology sphere.
Torrelli proved, that the following connection to local Bernstein-Sato polynomial exists.

\begin{theo}[Theorem 1.2 in \cite{Torrelli09}]
Let $Y\subset X$ be a hypersurface and $h\in \mathcal{O}_{X,x}$
a local equation of $Y$ at a point $y\in Y$. The following
conditions are equivalent:
\begin{enumerate}
\item $\mathcal{L}(Y,X)_y$ coincides with $H_{[Y]}^p
  (\mathcal{O}_X)_y$.
\item The reduced local Bernstein-Sato polynomial of $h$ has no integral root.
\end{enumerate}
\end{theo}

The proof of the theorem is based on a natural generalization
of a classical result due to Kashiwara which links the roots of
the $b$-function to some generators of $\mathcal{O}[\frac{1}{f}]
f^\alpha$, $\alpha\in\mathbb{C}$.

\begin{ex}
Let $Y$ be the affine variety in $X=\C^3$ defined by the polynomial
$f = z^7 + (x^2 z + y^3)(x^3 + y^2 z)$. The surface $Y$ has the origin
as its only singular point and thus the local $b$-function and and the global
one coincide. The only possible integral roots are $-2$ and $-1$.
Now consider $J_f$, the Jacobian ideal of $f$, cf. Remark \ref{reducedB}. Since the reduced
Bernstein-Sato polynomial is required, the ideal
$$
  \ann_{D[s]}(f^s) + D[s] \langle f, J_f, s+\alpha \rangle
$$
is used for checking rational roots, compare with Corollary \ref{cor1} (2).
We see that the above ideal is not the whole ring for $\alpha = 1$ and
hence the set of points $x\in Y$ such that $\mathcal{L}(Y,X)_x =
H_{[Y]}^p(\mathcal{O}_X)_x$ is $Y\setminus \{0\}$.

Using the implementation by Schulze \cite{Gmssinglib} (based on Gauss-Manin connection), the computation of the whole Bernstein-Sato polynomial took $1236$ seconds. While with our approach only 7 seconds were needed.
\end{ex}

\begin{remark}
Given $Y$ as above, the set of points $x\in Y$ for which the condition
$\mathcal{L}(Y,X)_x = H_{[Y]}^p(\mathcal{O}_X)_x$
is satisfied, defines an open set in $Y$ that can be effectively computed with the
stratification associated with the integral roots of the reduced
local $b$-functions, see the sequence of varieties (\ref{eqstrat}) below.
For instance, in Example \ref{exstrat2}, the open set is $V(f)\setminus V(y,z)$.
\end{remark}

\section{Stratification Associated with Local $b$-functions}\label{stratum}

From Theorem \ref{localBS}, one can find a stratification of $\C^n$ so that $b_{f,p}(s)$
is constant on each stratum. The first method for computing such stratification was suggested
by Oaku \cite{Oaku97b} (see also \cite{Oaku97a}, \cite{Oaku97c} and \cite{BO08} for further information). However, this method relies on the primary decomposition of commutative ideals. Following the ideas started in Section \ref{localcheckRoot}, we propose a new natural algorithm 
for computing such a stratification. At first, a stratification for each root of the global $b$-function is computed. Then one obtains a stratification, associated with the local $b$-function, notably without any primary ideal decomposition, see Example \ref{exstrat1} and \ref{exstrat2} below. 
We have created an experimental implementation, which was used for presented examples.  
The substitution of primary decomposition with elementary operations clearly decreases the 
total complexity of this algorithm.

This is a natural generalization of Proposition \ref{proplocalCR}.

\begin{theo}\label{stratif}
Let $\{ P_1(s), \ldots, P_k(s), f \}$ be a system of generators of $\ann_{D[s]}(f^s)
+ D[s]\langle f \rangle$ and consider the ideals $I_{\alpha,i} = \big(I:(s+\alpha)^i\, \big)
+ D[s] \langle s+\alpha \rangle$, for $\alpha$ root of $b_f(s)$ and $i=0,\ldots,m_\alpha-1$.
Then one has
$$
  m_\alpha(p) > i \ \Longleftrightarrow \ p\in V(I_{\alpha,i}\cap \C[{\bf x}]).
$$
\end{theo}

\begin{proof}
Repeat the same argument as in Corollary \ref{corAlg3} and proceed as in the proof of
Proposition \ref{proplocalCR}, using Lemmas \ref{lem1} and \ref{lem2} when necessary. 
\end{proof}

Using the notation of the above theorem, let $V_{\alpha,i}$ be the affine variety
corresponding to the ideal $I_{\alpha,i}\cap \C[{\bf x}]$. Then
\begin{equation}\label{eqstrat}
\emptyset =: V_{\alpha,m_{\alpha}} \subset V_{\alpha,m_{\alpha}-1} \subset \cdots
\subset V_{\alpha,0} \subset V_{\alpha,-1} := \C^n,
\end{equation}
and $m_\alpha(p) = i$ if and only if $p\in V_{\alpha,i-1}\setminus V_{\alpha,i}$.
We call this sequence the {\em stratification associated with the root $\alpha$}.
Let us see two examples\footnote{The hypersurfaces for examples have been taken from {\tt http://www.freigeist.cc/gallery.html}} to show how this result can be use to compute
a stratification associated with local $b$-functions.

\begin{ex}\label{exstrat1}
Consider $f = (x^2+9/4 y^2 + z^2 - 1)^3 - x^2 z^3 - 9/80 y^2 z^3 \in \C[x,y,z]$. The global
$b$-function is
$$
  b_f(s) = (s+1)^2 (s+4/3) (s+5/3) (s+2/3).
$$

Take $V_1 = V(x^2+9/4y^2-1,z)$, $V_2 = V(x,y,z^2-1)$ and
$V_3 = V(19x^2+1,171y^2-80,z)$. Then $V_2$ (resp. $V_3$) consists of two (resp. four)
different points and $V_3\subset V_1$, $V_1\cap V_3 = \emptyset$. The singular locus
of $f$ is union of $V_1$ and $V_2$. The stratification associated
with each root of $b_f(s)$ is given by

$$
\begin{array}{r c l c r c c c l}
\alpha &=& -1, & \ & \emptyset &\subset& V_1 \subset V(f) &\subset& \C^3\, ; \\
\alpha &=& -4/3, && \emptyset &\subset& V_1\cup V_2 &\subset& \C^3\, ; \\
\alpha &=& -5/3, && \emptyset &\subset& V_2 \cup V_3 &\subset& \C^3\, ; \\
\alpha &=& -2/3, && \emptyset &\subset& V_1 &\subset& \C^3 .
\end{array}
$$

From this, one can easily find a stratification of $\C^3$ into constructible sets such that $b_{f,p}(s)$ is constant on each stratum.

$$
b_{f,p}(s) =
\begin{cases}
1 & p\in\C^3\setminus V(f),\\
s+1 & p\in V(f)\setminus (V_1\cup V_2),\\
(s+1)^2 (s+4/3) (s+2/3) & p\in V_1\setminus V_3, \\
(s+1)^2 (s+4/3) (s+5/3) (s+2/3) & p\in V_3, \\
(s+1) (s+4/3) (s+5/3) & p\in V_2.
\end{cases}
$$
\end{ex}

\begin{ex}\label{exstrat2}
Let $M$ be the surface in $\C^3$ given by the polynomial
$f = x^2 y z + x y^2 + y^3 + y^3 z - x^2 z^2$. The global $b$-function is
$$
  b_f(s) = (s+7/6)^2 (s+5/6)^2 (s+1)^2 (s+4/3) (s+5/3) (s+3/2).
$$
The singular locus of $f$ is the union of $V(x,y)$ and $V(y,z)$. Now, consider the algebraic
set $V=V(x^2+4x, x z, y, z^2+z) = \{ 0, q_1, q_2\}$. Then the stratification associated with
each root of $b_f(s)$ is given by

$$
\begin{array}{r c l c r c c c l}
\alpha &=& -7/6, -5/6, & \ & \emptyset &\subset& V(x,y,z) \subset V(x,y)&\subset& \C^3\, ; \\
\alpha &=& -1, && \emptyset &\subset& V(y,z)\subset V(f) &\subset& \C^3\, ; \\
\alpha &=& -4/3, -5/3, && \emptyset &\subset& V(x,y,z)=\{0\} &\subset& \C^3\, ; \\
\alpha &=& -3/2, && \emptyset &\subset& V=\{0,q_1,q_2\} &\subset& \C^3 .
\end{array}
$$

Finally, we obtain a stratification associated with local $b$-functions and the corresponding
univariate polynomials.

$$
b_{f,p}(s) =
\begin{cases}
1 & p\in\C^3\setminus V(f),\\
s+1 & p\in V(f)\setminus (V(x,y)\cup V(y,z)),\\
(s+1) (s+7/6) (s+5/6) & p\in V(x,y)\setminus \{0, q_1\}, \\
(s+1)^2 & p\in V(y,z)\setminus \{0,q_2\}, \\
(s+1) (s+7/6) (s+5/6) (s+3/2) & p = q_1,\\
(s+1)^2 (s+3/2) & p = q_2,\\
b_f(s) & p = 0.
\end{cases}
$$
\end{ex}

\begin{remark}
Note that one can define a stratification associated with the roots of the
local $b$-functions, that is taking no multiplicities into account.
We have observed that our algorithm is especially useful and very fast for
computing this stratification. In particular, this is the case when each root
has multiplicity one. Finally, also observe that in any case the global $b$-function
is not actually needed, if an upper bound (or just a set containing the roots of
$b_f(s)$) is used instead.
\end{remark}

\begin{remark}
We see some common properties between the factorization of a \BS polynomial
with the so-called \textit{central character decomposition} by Levandovskyy \cite{LV3}.
In particular, for $b_f(s) = \prod_{\alpha\in A} (s-\alpha)^{m_\alpha}$, where $A\subset \mathbb{Q}$ is the set of roots of $b_f(s)$,  there is an algorithm
for computing the following direct sum decomposition of the module
\[
D[s]/(\ann_{D[s]} (f^s) + \langle f \rangle) \cong \bigoplus_{\alpha\in A} D[s]/(\ann_{D[s]} (f^s) + \langle f \rangle) : J(\alpha)^{\infty},
\]
where $J(\alpha) = \langle b_f(s)/(s-\alpha)^{m_\alpha} \rangle$. We plane to investigate this topic further and provide cyclic $D[s]$-modules, corresponding
to different strata.
\end{remark}

There is a very recent paper \cite{NN10} by Nishiyama and Noro, where the authors build a stratification without using primary decomposition. The authors use initial ideals with respect to weight vectors in computations, which is a classical (cf. \cite{SST00} alternative to the methods, utilizing annihilators $\ann_{D[s]} (f^s)$. In \cite{ALM09} there is a comparison of performance of both approaches for the computation of \BS polynomials. Notably, no method is clearly superior over another. Rather there are classes of examples, where the difference is very distinct. In particular, initial-based method score better results on hyperplane arrangements, while annihilator-based methods are better at complicated singularities, which are not hyperplane arrangements. A comparison of two methods for stratification is very interesting and it is an important task for the future. However, it seems to us that the method we presented will allow more thorough analysis of the algebraic situation due to the applicability of central character decomposition. At the moment it is not clear, whether such a decomposition exists for initial ideals.

\section{Other applications}
\label{otherApps}

\subsection{Bernstein-Sato Polynomials for Varieties}
\label{bsVar}


Let $f=(f_1,\ldots,f_r)$ be an $r$-tuple in $\K[x]^r$.
Denote by $\K\langle S \rangle$ the universal enveloping algebra $U(\mathfrak{g l}_{\,r})$,
generated by the set of variables $S=(s_{i j})$, $i,j=1,\ldots,r$ subject to relations:
$$
  [s_{i j}, s_{k l}] = \delta_{j k} s_{i l} - \delta_{i l} s_{k j}.
$$
Then, we denote by $D_n\langle S \rangle := D_n \otimes_{\K} \K\langle S \rangle$.
Consider a free $\K[x,s,\frac{1}{f}]$-module of rank one generated by the formal symbol
$f^s$ and denote it by $M = \K[x,s_{11},\ldots,s_{r r},\frac{1}{f_1\cdots f_r}]\cdot f^s$,
where $f^s = f_1^{s_{11}}\cdot \ldots \cdot f_r^{s_{r r}}$. The module $M$ has a natural structure of left $D_n \langle S \rangle$-module.
Denote by $\ann_{D_n\langle S\rangle}(f^s)$ the left ideal of all elements
$P(S) \in D_n\langle S \rangle$ such that $P(S)\bullet f^s = 0$, that is 
the {\em annihilator} of $f^s$ in $D_n \langle S \rangle$.


\begin{theo}[Budur, Musta\c{t}\v{a}, Saito \cite{BMS06}]
For every $r$-tuple $f=(f_1,\ldots,f_r)\in \K[x]^r$ there exists a non-zero
polynomial in one variable $b(s)\in \K[s]$ and $r$ differential operators
$P_1(S),\ldots,P_r(S)\in D_n\langle S \rangle$ such that
\begin{equation}\label{BSvariety}
\sum_{k=1}^r P_k(S) f_k \cdot f^s = b(s_{11}+\cdots+s_{r r}) \cdot f^s \ \in \ M.
\end{equation}
\end{theo}

The {\em Bernstein-Sato polynomial} $b_f(s)$ of
$f=(f_1,\ldots,f_r)$ is defined to be the monic polynomial of the lowest degree in the
variable $s$ satisfying the equation (\ref{BSvariety}). It can be verified that $b_f(s)$
is independent of the choice of a system of generators of $\langle f_1,\ldots,f_r\rangle$.
Then the \BS polynomials of $f$ can be computed as follows
$$
  (\ann_{D_n \langle S \rangle} (f^s) + \langle f_1,\ldots,f_r \rangle)\cap
  \K[s_{11}+\cdots+s_{r r}] = \langle b_f(s_{11}+\ldots+s_{r r})\rangle.
$$

In \cite{ALM09} an algorithm to find a system of generators of
$\ann_{D\langle S \rangle}(f^s)$ was given. Moreover, in computing
the intersection of an ideal with the univariate subalgebra an 
optimized algorithm (which uses linear algebra approach) was used.

The above formula together with Theorem \ref{mainresult} can be used to check 
rational roots of Bernstein-Sato polynomials also for affine algebraic varieties. 
Hence, following Corollary \ref{corAlg3}, a stratification associated with the local $b$-functions can be computed.

\subsection{A remark in Narv\'aez's paper}

In \cite{Narvaez08}, Narv\'aez introduces a polynomial denoted by $\beta(s)$
verifying $\beta(s) \ann_{D[s]}(f^s) \subseteq \ann_{D[s]}^{(1)}(f^s)$.
For all the examples treated in \cite{Narvaez08}, he was able to compute an operator $P'(s)\in D[s]$ such that $b_f(s)-P'(s)f \in \ann_{D[s]}^{(1)}(f^s)$.
The last example in the paper is quite involved and could not be computed by using any computer algebra system directly. An iterated process for finding approximations of involutive bases was used instead.

Indeed, for this propose the operator is not really needed, since 
$$
  b_f(s) - P(s)f\in \ann_{D[s]}^{(1)}(f^s) \ \Longleftrightarrow\
  b_f^{(1)}(s) = b_f(s)\ \Longleftrightarrow\
  b_f^{(1)}(s) \mid b_f(s),
$$
and thus after computing $b_f^{(1)}(s)$, one only has to check whether each root
of the latter polynomial is indeed a root of the $b$-function and the same with
the multiplicities.

By definition the following inclusions hold
$$
  \beta(s) \big(\ann_{D[s]}(f^s) + \langle f \rangle\,\big)\ \subset\
  \ann_{D[s]}^{(1)}(f^s) + \langle f \rangle\ \subset\
  \ann_{D[s]}(f^s) + \langle f \rangle.
$$
This implies that $b_f(s)\, |\, b_f^{(1)}(s)\, |\, \beta(s) b_f(s)$. Additionally, if $\beta(s)$ divides $b_f(s)$, then the polynomials $b^{(1)}_f(s)$ and $b_f(s)$ both have the same roots and the above condition is equivalent to $m_\alpha(b^{(1)}_f(s)) = m_\alpha(b_f(s))$ for every root $\alpha$ of $\beta(s)$.



\begin{ex}
Let $f=(x_1 x_3 + x_2)(x_1^7 - x_2^7)$ be the last example from \cite{Narvaez08}. The Bernstein-Sato polynomial and the polynomial $\beta(s)$ are respectively
\begin{eqnarray*}
b_f(s) &=& (s + 1)^3 (s + 3/4) (s + 3/8) (s + 9/8) (s + 1/4) (s + 7/8) (s + 1/2) (s + 5/8),\\
\beta(s) &=& (s + 3/4) (s + 5/8) (s + 1/2) (s + 3/8) (s+1/4).
\end{eqnarray*}

Now one only has to check that all roots of $\beta(s)$ have multiplicity $1$ as a root of $b_f^{(1)} (s)$. This can be done using Theorem \ref{mainresult} with $I = \ann^{(1)}_{D[s]} (f^s) + \langle f \rangle$. Using this approach the computations are trivial, less than $5$ seconds.
\end{ex}

\section{Conclusion and Further Work}
\label{conclusion}

As we have demonstrated, the family of \texttt{checkRoot} algorithms (implemented 
in the library \texttt{dmod.lib} of {\sc Singular}) has many useful applications in the realm of $D$-modules. 
Nowadays, it is the only method that allows one to obtain some roots of the $b$-function without computing the whole Bernstein-Sato polynomial. The latter is often infeasible despite all the
recent progress in computational $D$-module theory.

We emphasize, that presented techniques are elementary (by utilizing
the principal ideal domain of the center $\K[s]$ of $D_n[s]$) but very powerful from computational point of view. Many intractable examples and conjectures could be treated with this new method, as we have partially illustrated.
Moreover, a stratification associated with the local $b$-functions can be obtained
without primary decomposition \cite{Oaku97b} as in the very recent \cite{NN10}. 
It is very interesting to study these algorithms further and compare our approach
with the one of \cite{NN10}.

Unfortunately, these techniques cannot be generalized for Bernstein-Sato ideals, since such ideals lie in $\K[s_1,\ldots,s_m]$ for $m\geq 2$.


We have demonstrated that one can use the idea of {\tt checkRoot} for checking rational roots of $b$-function of a holonomic ideal with respect to a weight vector \cite{SST00}.
This gives an easier method for computing, among other, integral roots of such $b$-functions, if an upper bound is known in advance. In this context, it would be very interesting to have a version of Kashiwara's result for some holonomic ideals and certain weights, since many algorithms in $D$-modules theory are based on integrations and restrictions which need minimal/maximal roots.



\section*{Acknowledgements}


We would like to thank Francisco Castro-Jim\'{e}nez, Jos\'{e}-Mar\'{i}a Ucha, Gert-Martin Greuel,
Enrique Artal and Jos\'{e}-Ignacio Cogolludo for their constant support and motivation in our
work over years.

We wish to thank to Uli Walther and Daniel Andres for fruitful discussions on the topics of
this article.


The authors thank to the DFG Graduiertenkolleg "Hierarchie und Symmetrie in mathematischen Modellen" at RWTH Aachen, Germany and projects MTM2007-67908-C02-01, FQM-333 and ``E15 Grupo Consolidado Geometr\'ia (DGA)'' of Spain for the partial financial support.


\bibliography{./checkRootRef}
\bibliographystyle{plain}


%
%
%
%
\end{document}